\documentclass{amsart}
     
\newtheorem{theorem}{Theorem}[section]
\newtheorem{lemma}[theorem]{Lemma}

\theoremstyle{definition}
\newtheorem{definition}[theorem]{Definition}
\newtheorem{example}[theorem]{Example}

\theoremstyle{remark}
\numberwithin{equation}{section}

\def\a{{\alpha}}
\def\b{{\beta}}
\def\g{{\gamma}}

\def\ep{{\eta}}
\def\z{{\zeta}}

\def\f{{\varphi}}
\def\na{{\nabla}}

\def\p{{\psi}}

\def\G{{\Gamma}}
\def\O{{\Omega}}
\def\D{{\Delta}}

\def\ff{{{\mathcal F}}}
\def\gg{{{\mathcal G}}}

\def\xxx{{{\bf x}}}

\def\JJJ{{{\bf J}}}

\def\qqq{{{\bf q}}}

\def\RN{{{\bf R}^N}}

\def\9{{\ \hbox{in}\ \O}}
\def\1{{\ \hbox{on}\ \G_1}}
\def\2{{\ \hbox{on}\ \G_2}}
\def\3{{\ \hbox{on}\ \G_3}}
\def\0{{\ \hbox{on}\ \G}}

\def\pa{{\partial}}

\begin{document}

\title[An application of a theorem of G. Zwirner]{An application of a theorem of G. Zwirner to a class of non-linear elliptic systems in divergence form}
\author{Giovanni Cimatti}
\address{Department of Mathematics, Largo Bruno
  Pontecorvo 5, 56127 Pisa Italy}
\email{cimatti@dm.unipi.it}






\begin{abstract}
A theorem on the solutions of the problem $U'(w)=\g F(U(w),w),\quad U(w_1)=u_1,\ U(w_2)=u_2$ is applied for finding the functional solutions of the system of partial differential equations
\begin{equation*}
\na\cdot(a(u,w)\na u)=0,\ u=u_1\1,\quad u=u_2\2,\quad\frac{\pa u}{\pa n}=0\3
\end{equation*} 
\begin{equation*}
\na\cdot(b(u,w)\na w)=0,\ w=w_1\1,\quad w=w_2\2,\quad\frac{\pa w}{\pa n}=0\3.
\end{equation*}
The problem of existence and uniqueness of solutions is considered. 
\end{abstract}

\maketitle

\section{Introduction}
The problem of finding the solutions of the ordinary differential equation
\begin{equation}
\label{1_1}
U'(w)=\g F(U(w),w)
\end{equation}
which satisfy the two conditions

\begin{equation}
\label{2_1}
U(w_1)=u_1,\quad U(w_2)=u_2,\quad w_2>w_1
\end{equation}
was the object of several papers mainly of the Italian  and Japanese school. We quote in particular \cite{HN}, \cite{Z}, \cite{C}, \cite{S}, \cite{ZW}, \cite{S1}, \cite{ZW1}. In this paper we show that the theorem given by G. Zwirner in \cite{ZW} on the existence and uniqueness for problem (\ref{1_1}), (\ref{2_1}) can be used to find a class of solutions, physically relevant, of the boundary value problem

\begin{equation}
\label{3_1}
\na\cdot(a(u,w)\na u)=0\quad\9
\end{equation} 

\begin{equation}
\label{4_1}
u=u_1\1,\quad u=u_2\2,\quad\frac{\pa u}{\pa n}=0\3
\end{equation} 

\begin{equation}
\label{5_1}
\na\cdot(b(u,w)\na w)=0\quad\9
\end{equation} 

\begin{equation}
\label{6_1}
w=w_1\1,\quad w=w_2\2,\quad\frac{\pa w}{\pa n}=0\3,\quad w_2>w_1
\end{equation} 
where $\O$ is an open and bounded subset of $\RN$ with boundary $\G$ divided into three parts $\G_1$, $\G_2$ and $\G_3$. $u_1$, $u_2$ are arbitrary constants, whereas $w_1$, $w_2$ are constants with the restriction $w_2>w_1$\ \footnote{The assumption $w_2>w_1$, (or, more generally, $w_2\neq w_1$) is essential to make problem (\ref{1_1}), (\ref{2_1}) meaningful. On the other hand, if we assume $w_1=w_2=\bar w$ the problem (\ref{3_1})-(\ref{6_1}) is immediately uncoupled. In fact, from (\ref{5_1}) and (\ref{6_1}) we have $w(\xxx)=\bar w$, under the sole assumption $b(u,w)>0$. Substituting this value of $w$ in (\ref{3_1}), the problem (\ref{3_1})-(\ref{4_1}) can be solved using the Kirchhoff transformation}.

When $N=3$ the problem (\ref{3_1})-(\ref{6_1}) has a simple physical interpretation. For, let $u(\xxx)$, $\xxx\in\O$ represent the temperature and $w(\xxx)$ the concentration of a substance in a liquid at rest which occupies $\O$. Suppose that on $\G_1$ and $\G_2$ the temperature $u$ and the concentration $w$ are kept fixed at the two constant values $u_1$, $u_2$ and $w_1$, $w_2$ respectively, whereas $\G_3$ is the part of the boundary of $\O$ which is thermally insulated and impermeable to the substance dissolved in the fluid. By the Fourier's law we have for the density of heat flow $\qqq=-a(u,w)\na u$ and for the density of molecular mass flow  $\JJJ=-b(u,w)\na w.\ \footnote{In certain situations the dependence of $a$ and $b$ on $u$, $w$ can be quite strong.}$ In absence of sources of heat and mass we have $\na\cdot\qqq=0,\ \na\cdot\JJJ=0$ i.e. (\ref{3_1}) and (\ref{5_1}).

\section{Existence and uniqueness of functional solutions}

 We assume that the boundary of $\O$ has a degree of regularity which makes solvable the mixed problem

\begin{equation}
\label{1_3}
\D z=0,\quad z=0\ \1,\quad z=1\2\quad,\quad\frac{\pa z}{\pa n}=0\ \3.
\end{equation} 
We are interested in the functional solutions of problem (\ref{3_1})-(\ref{6_1}) according to the following

\begin{definition}
A classical solution $(u(\xxx),w(\xxx)$) of problem (\ref{3_1})-(\ref{6_1}) is termed functional if a function $U(t)\in C^1([w_1,w_2])$ exists such that $u(\xxx)=U(w(\xxx))$.
\end{definition}

\begin{example}
 Let us consider the special case of (\ref{3_1})-(\ref{6_1}) in which 

\begin{equation}
\label{2_4}
a(u,w)=b(u,w),\quad a(u,w)\geq a_0>0.
\end{equation} 
We claim that every classical solution $(u(\xxx),w(\xxx))$ of (\ref{3_1})-(\ref{6_1}) is a functional solution with respect to the function 

\begin{equation*}
U(t)=\a t+\b,\quad \a=\frac{u_2-u_1}{w_2-w_1},\quad \b=\frac{u_1w_2-w_1u_2}{w_2-w_1}.
\end{equation*} 
For, let $(u(\xxx),w(\xxx))$ be any solution of (\ref{3_1})-(\ref{6_1}) and define $\z(\xxx)=u(\xxx)-(\a w(\xxx)+\b)$. We have

\begin{equation}
\label{1_5}
\na\cdot(a(u,w)\na\z)=0\ \9,\quad\z=0\ \1,\ \z=0\2\quad,\ \frac{\pa \z}{\pa n}=0\ \3.
\end{equation} 
Multiplying (\ref{1_5}) by $\z$ and integrating by parts over $\O$ we have, in view of (\ref{2_4}), $\z(\xxx)=0$. Hence  $(u(\xxx),w(\xxx))$ is a functional solution since we have $u(\xxx)=U(w(\xxx))$. For other applications of the functional solutions of systems of partial differential equations in divergence form we refer to \cite{GC}and \cite{GC1}.
\end{example}
\vskip .1cm Associated with the problem (\ref{3_1})-(\ref{6_1}) we consider the two-point problem

\begin{equation}
\label{1_6}
U'(w)=\g \frac{b(U(w),w)}{a(U(w),w)}
\end{equation}

\begin{equation}
\label{2_6}
U(w_1)=u_1,\quad U(w_2)=u_2,\quad w_2>w_1.
\end{equation}
To this problem we can apply the following theorem (see \cite{ZW} for the proof).

\begin{theorem}
Let $F(U,w)$ be measurable with respect to $w$ and continuous with respect to $U$ in the rectangle $R=\{w_1\leq w\leq w_2,\ u_1\leq U\leq u_2\}$, $w_1<w_2$. Assume that there exist two functions $q(w),\ p(w)\in L^1(w_1,w_2)$ such that

\begin{equation*}
p(w)\leq F(U,w)\leq q(w)
\end{equation*}

\begin{equation*}
p(w)\geq 0,\quad \int_{w_1}^{w_2}p(t)dt>0.
\end{equation*}
Then the problem
\begin{equation}
\label{3_7}
 U'(w)=\g F(U(w),w),\quad U(w_1)=u_1,\quad U(w_2)=u_2,
\end{equation}
 in the unknown $\g$ (a real number) and $U(w)$, has at least one solution absolutely continuous in $[w_1,w_2]$. If $F(U,w)\in C^k( R)$ then $u(t)\in C^{k+1}([w_1,w_2])$. Moreover, if $F(U,w)$ satisfies a Lipschitz condition in $R$ with respect to $U$ the solution of (\ref{3_7}) is unique.\ \footnote{Other criteria which guarantee the uniqueness of the solution can be found in \cite{S1}.}
\end{theorem}
\vskip .05 cm
The link between the problem (\ref{3_1})-(\ref{6_1}) and the problem (\ref{1_6}), (\ref{2_6}) is established in the theorems below using the following elementary

\begin{lemma}
Let $w(\xxx)\in C^0(\bar\O)$ and

\begin{equation*}
\min_{\bar\O}w(\xxx)=w_1\leq w(\xxx)\leq w(\xxx)\leq w_2=\max_{\bar\O}w(\xxx).
\end{equation*}
Assume $\ff(t),\ \gg(t)\in C^0([w_1,w_2])$, then, if

\begin{equation}
\label{2_9}
\ff(w(\xxx))=\gg(w(\xxx)),\quad \xxx\in\bar\O,
\end{equation}
we have, for all $w\in[w_1,w_2]$,

\begin{equation*}
\ff(w)=\gg(w).
\end{equation*}
\end{lemma}

\begin{proof}
Assume $w^*\in [w_1,w_2]$. There exists $\xxx^*\in\bar\O$ such that $w(\xxx^*)=w^*$. Hence, by (\ref{2_9}),

\begin{equation}
\label{4_9}
\ff(w^*)=\ff(w(\xxx^*))=\gg(w(x^*))=\gg(w^*).
\end{equation}
\end{proof}

\begin{theorem}
Let $w_2>w_1$ and $R=\{(u,w);\ u_1\leq u\leq u_2,\ w_1\leq w\leq w_2\}$. Assume $b(u,w)\in C^0(R)$ and

\begin{equation}
\label{1_10}
a(u,w), \quad b(u,w)>0\quad \hbox{in}\quad R.
\end{equation}
Let $(u(\xxx),w(\xxx))$ be a functional solution of the problem

\begin{equation}
\label{1_11}
\na\cdot(a(u,w)\na u)=0\quad\9
\end{equation} 

\begin{equation}
\label{2_11}
u=u_1\1,\quad u=u_2\2,\quad\frac{\pa u}{\pa n}=0\3
\end{equation} 

\begin{equation}
\label{3_11}
\na\cdot(b(u,w)\na w)=0\quad\9
\end{equation} 

\begin{equation}
\label{4_11}
w=w_1\1,\quad w=w_2\2,\quad\frac{\pa w}{\pa n}=0\3,
\end{equation} 
then the function $U(w)$ entering in the definition of functional solution solves the two point-problem

\begin{equation}
\label{5_11}
U'(w)=\frac{b(U(w),w)}{a(U(w),w)}
\end{equation}

\begin{equation}
\label{6_11}
U(w_1)=u_1,\quad U(w_2)=u_2,\quad w_2>w_1.
\end{equation}
\end{theorem}

\begin{proof}
Let $(u(\xxx),w(\xxx))$be a functional solution of (\ref{1_11})-(\ref{4_11}). By (\ref{3_11}) the maximum principle \cite{PW} implies

\begin{equation}
\label{1_12}
w_1\leq w(\xxx)\leq w_2.
\end{equation}
Moreover, by assumption $u(\xxx)=U(w(\xxx))$.
Define

\begin{equation}
\label{1_13}
\theta(w)=\int_{w_1}^{w} a(U(t),t)U'(t)dt,\quad \p(w)=\int_{w_1}^{w} b(U(t),t)dt
\end{equation}
and

\begin{equation}
\label{2_13}
 \Theta(\xxx)=\theta(w(\xxx)),\quad \Psi(\xxx)=\p(w(\xxx)).
\end{equation}
We have $\na\Theta=a(u,w)\na u,\quad \na\Psi=b(u,w)\na w$. On the other hand, $(u(\xxx),w(\xxx))$ solves (\ref{1_11})-(\ref{4_11}), thus we have

\begin{equation*}
\D\Theta=0\ \9,\quad \Theta=0\ \1
\end{equation*}

\begin{equation*}
 \Theta=\theta(w_2)\ \2,\quad \frac{\pa\Theta}{\pa n}=0\ \3
\end{equation*}

\begin{equation*}
\D\Psi=0\ \9,\quad \Psi=0\ \1
\end{equation*}

\begin{equation*}
 \Psi=\psi(w_2)\ \2,\quad \frac{\pa\Theta}{\pa n}=0\ \3.
\end{equation*}
By (\ref{1_10}) we have $\psi(w_2)\neq 0$. Let $z(\xxx)$ be the solution of the problem (\ref{1_3}). We obtain $\Theta(\xxx)=\theta(w_2) z(\xxx)$ and $\Psi(\xxx)=\psi(w_2) z(\xxx)$. Hence

\begin{equation}
\label{1_15}
 \Theta(\xxx)=\g\Psi(\xxx),\quad \g=\frac{\theta(w_2)}{\psi(w_2)}.
\end{equation}
From (\ref{1_13}), (\ref{2_13}) and (\ref{1_15}) we have

\begin{equation}
\label{3_15}
 \int_{w_1}^{w(\xxx)}a(U(t),t)U'(t)dt=\g\int_{w_1}^{w(\xxx)}b(U(t),t)dt.
\end{equation}
Applying Lemma 1.4 with 

\begin{equation*}
\ff(t)= \int_{w_1}^{t}a(U(\ep),\ep)U'(\ep)d\ep, \quad \gg(t)=\int_{w_1}^{t}b(U(\ep),\ep)d\ep
\end{equation*}
by (\ref{3_15}) we have

\begin{equation*}
 \int_{w_1}^{w}a(U(t),t)U'(t)dt =\g\int_{w_1}^{w}b(U(t),t)dt.
\end{equation*}
Hence

\begin{equation*}
 a(U(w),w)U'(w)=\g b(U(w),w)
\end{equation*}
and (\ref{5_11}) holds.
Moreover, also the boundary conditions (\ref{6_11}) are verified.
\end{proof}
Vice-versa we have

\begin{theorem}
Assume (\ref{1_10}), then to every solution $U(w)$ of class $C^1([w_1,w_2])$ of the problem

\begin{equation}
\label{1_17}
U'(w)=\g\frac{b(U(w),w)}{a(U(w),w)},\quad U(w_1)=u_1,\quad U(w_2)=u_2,\quad w_2>w_1
\end{equation}
there corresponds a functional solution of the problem (\ref{1_11})-(\ref{4_11}).
\end{theorem}

\begin{proof}
Let $U(t)$ be a solution of (\ref{1_17}) and consider the non-linear elliptic problem

\begin{equation}
\label{1_18}
\na\cdot(b(U(w),w)\na w)=0\quad\9
\end{equation} 

\begin{equation}
\label{2_18}
w=w_1\1,\quad w=w_2\2,\quad\frac{\pa w}{\pa n}=0\3.
\end{equation}
There exists one and only one solution of (\ref{1_18}), (\ref{2_18}). For, let us define

\begin{equation*}
\p(w)=\int_{w_1}^w b(U(t),t)dt.
\end{equation*}
By (\ref{1_10}) $\p$ maps one-to-one $[w_1,w_2]$ onto $[0,\p(w_2)]$. Hence, if we define $\f(\xxx)=\p(w(\xxx))$, the problem (\ref{1_18}), (\ref{2_18}) can be restated as

\begin{equation}
\label{1_19}
\D\f=0\ \9,\quad \f=0\ \1
\end{equation}

\begin{equation}
\label{2_19}
 \f=\p(w_2)\ \2,\quad \frac{\pa\f}{\pa n}=0\ \3.
\end{equation}
By (\ref{1_3}) the solution of (\ref{1_19}) and (\ref{2_19}) exists and is unique and $w(\xxx)=\p^{-1}(\f(\xxx))$ gives the unique solution of (\ref{1_18}), (\ref{2_18}). Define now

\begin{equation*}
 u(\xxx)=U(w(\xxx)).
\end{equation*}
Thus (\ref{1_18}) can be written

\begin{equation*}
\na\cdot(b(u,w)\na w)=0\quad\9.
\end{equation*} 
Setting $w=w(\xxx)$ in (\ref{1_17}) we obtain

\begin{equation*}
 a(U(w(\xxx)),w(\xxx))U'(w(\xxx))=\g b(U(w(\xxx)),w(\xxx))
\end{equation*}
and also

\begin{equation*}
 a(U(w(\xxx)),w(\xxx))U'(w(\xxx))\na w=\g b(U(w(\xxx)),w(\xxx))\na w
\end{equation*}
and, by (\ref{1_18}),

\begin{equation*}
\na\cdot(a(u,w)\na u)=0\quad\9.
\end{equation*} 
On the other hand, the functions $(u(\xxx),w(\xxx))$ just defined satisfies also the boundary conditions (\ref{2_11}) and (\ref{4_11}).
\end{proof}
This proof shows that the problem (\ref{1_11})-(\ref{4_11}) is solvable (i) if we can solve the linear problem (\ref{1_19}), (\ref{2_19}), which in turn is immediately reducible to (\ref{1_3}) which contains the ``geometric'' part, (ii) a solution of problem (\ref{5_11}), (\ref{6_11}) is known. This last solution contains the non-linear features of the original problem (\ref{3_1})-(\ref{6_1}) if we limit ourselves to consider functional solutions.

\vskip .3cm
 The uniqueness of the functional solutions of problem (\ref{3_1})-(\ref{6_1}) is also a consequence of the uniqueness for problem (\ref{1_1}), (\ref{2_1}). In fact we have

\begin{theorem}
Let (\ref{1_10}) hold. If the problem

\begin{equation}
\label{1_21}
U'(w)=\g\frac{b(U(w),w)}{a(U(w),w)},\quad U(w_1)=u_1,\quad U(w_2)=u_2
\end{equation} 
has a unique solution also the corresponding functional solution of

\begin{equation}
\label{5_21}
\na\cdot(a(u,w)\na u)=0\quad\9
\end{equation} 

\begin{equation}
\label{6_21}
u=u_1\1,\quad u=u_2\2,\quad\frac{\pa u}{\pa n}=0\3
\end{equation} 

\begin{equation}
\label{7_21}
\na\cdot(b(u,w)\na w)=0\quad\9
\end{equation} 

\begin{equation}
\label{8_21}
w=w_1\1,\quad w=w_2\2,\quad\frac{\pa w}{\pa n}=0\3
\end{equation} 
is unique in the class of functional solutions.
\end{theorem}

\begin{proof}
Let, by contradiction, $(u^*,w^*)$, $(u^{**},w^{**})$ be two functional solutions of problem (\ref{5_21})-(\ref{8_21}). We have

\begin{equation}
\label{1_22}
u^*(\xxx)=U^*(w^*(\xxx)),\quad u^{**}(\xxx)=U^{**}(w^{**}(\xxx)).
\end{equation} 
$U^*(w)$ and $U^{**}(w)$ are both solutions of the problem (\ref{1_21}). Thus $U^*(w)=U^{**}(w)$. Let us define

\begin{equation*}
\p^*(w)=\int_{w_1}^w b(U^*(t),t)dt,\quad \p^{**}(w)=\int_{w_1}^w b(U^{**}(t),t)dt
\end{equation*} 
and

\begin{equation*}
\Psi^*(\xxx)=\psi^*(w^*(\xxx)),\quad \Psi^{**}(\xxx)=\psi^{**}(w^{**}(\xxx)).
\end{equation*} 
We have $\p^*(w_2)=\p^{**}(w_2)$, therefore $\Psi^*(\xxx)$ and $\Psi^{**}(\xxx)$ are both solutions of the problem

\begin{equation}
\label{3_23}
\D\f=0\ \9,\quad \f=0\ \1,
\end{equation}

\begin{equation}
\label{4_23}
\f= \p^*(w_2)\ \2,\quad \frac{\pa\f}{\pa n}=0\ \3
\end{equation}
which has a unique solution. Hence $\Psi^*(\xxx)=\Psi^{**}(\xxx)$ and we have  $\p^*(w)=\p^{**}(w)$ by Lemma 1.3. This in turn implies

\begin{equation}
\label{5_23}
w^*(\xxx)=(\p^*)^{-1}(\f(\xxx))=(\p^{**})^{-1}(\f(\xxx))=w^{**}(\xxx)
\end{equation}
and

\begin{equation}
\label{6_23}
u^*(\xxx)=U^*(w^*(\xxx))=U^{**}(w^{**}(\xxx))=u^{**}(\xxx).
\end{equation}

\end{proof}
We summarize our results in the following

\begin{theorem}
Let $\frac{b(U,w)}{a(U,w)}$ be of class $C^1$ in the rectangle $R=\{w_1\leq w\leq w_2,\ u_1\leq U\leq u_2\}$, $w_1<w_2$. Assume (\ref{1_10}) and that there exist two functions $q(w),\ p(w)\in L^1(w_1,w_2)$ such that

\begin{equation*}
0\leq p(w)\leq \frac{b(U,t)}{a(U,t)}\leq q(w),\quad \int_{w_1}^{w_2}p(t)dt>0.
\end{equation*}
Then the problem (\ref{3_1})-(\ref{6_1}) has at least one functional solution. Moreover, if $\frac{b(U,w)}{a(U,w)}$ satisfies a Lipschitz condition in $R$ with respect to $U$ the solution of (\ref{3_1})-(\ref{6_1}) is unique in the class of functional solutions.
\end{theorem}

\noindent{\bf Compliance with ethical standard}
\vskip .3cm
\noindent{\bf Conflict of interest.} The author declares that he has no conflicts of interest.

\bibliographystyle{amsplain}

\end{document}